\documentclass[11pt]{article}
\usepackage{amsmath}
\usepackage{amsfonts}
\usepackage{amssymb}
\usepackage{hyperref} 
\usepackage{amsthm}
\usepackage[symbol]{footmisc}

\hypersetup{colorlinks,
citecolor=blue,
filecolor=black,
linkcolor=black,
urlcolor=black}
\usepackage{graphics,epstopdf}
\usepackage[pdftex]{graphicx}
\usepackage{newlfont}\newlength{\defbaselineskip}
\usepackage[left=1in,right=1in,top=1in,bottom=0.8in,footskip=0.25in]{geometry}
\usepackage{setspace}
\allowdisplaybreaks
\newtheorem{theorem}{Theorem}[section]

\newtheorem{remark}{Remark}[section]

\numberwithin{equation}{section}

\usepackage[table]{xcolor}
\newtheorem{corollary}{Corollary}[section]

\usepackage{fancyhdr}
\usepackage{graphics}
\usepackage{color}
\usepackage{rotating}
\usepackage{fancybox}
\begin{document}
\begin{center}
\textbf{\Large{Euler-type integrals for the generalized hypergeometric matrix function }}
\end{center}
\begin{center}
{\bf  Ankit Pal\textsuperscript{a}, Kiran Kumari\textsuperscript{b}}\\ 
ankit.pal@vitbhopal.ac.in, kiran.kumariprajapati@gmail.com,  \\
\textsuperscript{a}Department of Mathematics, School of Advanced Sciences \& Languages\\
VIT Bhopal University, Sehore-466 114, Madhya Pradesh, India.\\
\textsuperscript{b}Department of Mathematics and Humanities,\\
Sardar Vallabhbhai National Institute of Technology, Surat-395 007, Gujarat, India. 
\end{center}
\rule{\textwidth}{0.2pt}
\textbf{Abstract:}
In this paper, we investigate the Euler-type integral representations for the generalized hypergeometric matrix function and develop some transformations in terms of hypergeometric matrix functions. Furthermore, unit and half arguments have been provided for several particular cases.\\[0.2cm]
\noindent
{\bf Keywords:} Generalized hypergeometric matrix function, Matrix functional calculus, Gamma and Beta matrix function.\\[0.2cm]
\noindent\textbf{AMS Subject Classification (2010):} 15A15, 33C05, 33C20, 44A99.\\[0.1cm]
\rule{\textwidth}{0.2pt}
\section{Introduction and Preliminaries}
In recent years, mathematical analysis has paid a lot of attention to the theory of special matrix functions and polynomials. The matrix analogue of hypergeometric functions play a significant role in the area of applied and pure analysis. These functions appears in the study of statistics \cite{Constantine}, probability theory \cite{Seaborn} and Lie theory \cite{James}. The matrix analogue of the Gauss hypergeometric function was introduced by J\'odar and Cort\'es \cite{Jodar1}. They studied the integral representation and matrix differential equation. Later, Abdalla \cite{Abdalla,Abdalla1}, Dwivedi and Sahai \cite{Dwivedi,Dwivedi1} studied the various properties like integral and differential relations, finite sum formulas, generating functions of special matrix functions. Particularly, these functions play a vital role in solving numerous problems of mathematical physics, engineering and mathematical sciences.\\

Throughout this work, we consider the complex space $\mathbf{C}^{R\times R}$ of complex matrices of common order $R$. For any matrix $P\in\mathbf{C}^{R\times R}$, $\sigma(P)$ is the spectrum of $P$ and
\begin{align}
a(P) = \max\left\lbrace \Re(z) : z\in\sigma(P)\right\rbrace, \quad b(P) = \min\left\lbrace \Re(z) : z\in\sigma(P)\right\rbrace,
\end{align}
where $a(P)$ is the spectral abscissa of $P$ and $b(P) = - a(-P)$. A Hermitian matrix $P$ in $\mathbf{C}^{R\times R}$ is a positive stable matrix if $\Re(\lambda)>0$ for all $\lambda\in\sigma(P)$, where $\sigma(P)$ is the set of all eigenvalues of $P$ or spectrum of $A$ and its two-norm is given by
\begin{align*}
||P|| = \sup_{x\neq 0} \frac{||Px||_2}{||x||_2} = \max\left\lbrace \sqrt{\lambda}: \lambda \in \sigma(P^*P) \right\rbrace,
\end{align*}
where for any vector $x\in\mathbf{C}^r$, $||x||_2 = (x^*x)^{1/2}$ is the Euclidean norm of $x$ and $P^*$ denotes the conjugate transpose of $P$. $I$ and $\mathbf{0}$ stands for the identity matrix and null matrix in $\mathbf{C}^{R\times R}$, respectively. Taking into account the Sch\"{u}r decomposition of a matrix $P$ \cite{Golud}, we have 
\begin{align*}
\left\| e^{Pt} \right\| \leq e^{t a(P)} \sum_{u=0}^{r-1} \frac{\left(\left\| P \right\| r^{1/2}t\right)^u}{u!} \quad (t\geq 0),
\end{align*}
which yields
\begin{align*}
\left\| t^P \right\| \leq \left\| e^{P \ln t} \right\|  \leq t^{a(P)} \sum_{u=0}^{r-1} \frac{\left(\left\| P \right\| r^{1/2} \ln t\right)^u}{u!} \quad (t\geq 1).
\end{align*}
If $f(z)$ and $g(z)$ are holomorphic functions of the complex variable $z$, which are defined in an open set $\Omega$ of the complex plane, and $P$ is a matrix in $\mathbf{C}^{R\times R}$ with $\sigma(P)\subset \Omega$, then from the properties of the matrix functional calculus \cite{Dunford}, it follows that
\begin{align*}
f(P)g(P) = g(P)f(P).
\end{align*} 
Furthermore, if $Q\in \mathbf{C}^{R\times R}$ with $\sigma(Q)\subset \Omega$, and if $PQ=QP$, then
\begin{align}
f(P)g(Q) = g(Q)f(P).
\end{align}
The logarithmic norm of a matrix $P\in \mathbf{C}^{R\times R}$ is defined as (see \cite{Cortes,HuGD}),
\begin{align}
\mu(P) = \lim_{k\to 0} \frac{\left\| I+kP\right\|-1}{k} = \max \left\lbrace z \ | \ z \in \sigma\left(\frac{P+P^*}{2}\right)\right\rbrace
\end{align}
Let the number $\tilde{\mu}(P)$ such that
\begin{align}
\tilde{\mu}(P) = -\mu(-P) = \min \left\lbrace z \ | \ z \in \sigma\left(\frac{P+P^*}{2}\right)\right\rbrace.
\end{align}
The reciprocal gamma function $\Gamma^{-1}(z) = 1/\Gamma(z)$ is an entire function of the complex variable $z$. The image of $\Gamma^{-1}(z)$ acting on $P$, denoted by $\Gamma^{-1}(P)$, is a well defined matrix. If $P+nI$ is invertible for all integers $n\geq 0$, then the reciprocal gamma function is defined as (see \cite{Jodar3}),
\begin{align}
\Gamma^{-1}(P) = P(P+I)\dots (P+(n-1)I) \ \Gamma^{-1}(P+nI), \ \ n\geq 1.
\end{align}
By applications of the matrix functional calculus, the Pochhammer symbol \cite{Jodar3} for $P\in \mathbf{C}^{R\times R}$ is given by
\begin{align}\label{PCH1}
(P)_m = \begin{cases} 
     I, \;\;\;\;\;\;\;\;\;\; \hspace{4cm}          \text{if} \; m=0\\
      P(P+I)\dots(P+(m-1)I),   \hspace{0.8cm} \text{if} \; m\geq 1,
   \end{cases}
\end{align} 
which gives
\begin{align}\label{PCH2}
(P)_m = \Gamma^{-1}(P) \ \Gamma(P+mI), \quad m\geq 1.
\end{align}
If $P\in \mathbf{C}^{R\times R}$ is a positive stable matrix and $m\geq 1$ is an integer, then the gamma matrix function can be represented in the following limit form as \cite{Jodar1}:
\begin{align}
\Gamma(P) = \lim_{m\to \infty} (m-1)! \ (P)_m^{-1} \ m^P.
\end{align}
Let $P$ and $Q$ be two positive stable matrices in $\mathbf{C}^{R\times R}$. The gamma matrix function $\Gamma(P)$ and the beta matrix function $B(P,Q)$ have been defined in \cite{Jodar1,Jodar3}, as follows
\begin{align}
\Gamma(P) = \int_0^\infty e^{-t} \ t^{P-1} dt; \quad t^{P-1} = \exp((P-I)\ln t),
\end{align}
and
\begin{align}
B(P,Q) = \int_0^1 t^{P-1} \ (1-t)^{Q-1} dt.
\end{align}
Let P and Q be commuting matrices in $\mathbf{C}^{R\times R}$ such that the matrices $P+nI, Q+nI$ and $P+Q+nI$ are invertible for every integer $n\geq 0$, then according to \cite{Jodar3}, we have
\begin{align}
B(P,Q) = \Gamma(P) \Gamma(Q) \left[ \Gamma(P+Q)\right]^{-1}.
\end{align}

\section{Main results}

Dwivedi and Sahai \cite{Dwivedi} introduced a natural generalization of hypergeometric matrix function called generalized hypergeometric matrix function and defined by
\begin{align}\label{sahai1}
{}_pF_q\left(P_1,\dots,P_p;Q_1,\dots,Q_q;z\right) = \sum_{m=0}^\infty\frac{(P_1)_m\dots(P_p)_m(Q_1)_m^{-1}\dots (Q_q)_m^{-1}}{m!} z^m,
\end{align}
where $P_i,Q_j \in \mathbf{C}^{R\times R}, 1\leq i \leq p,1 \leq j\leq q$, such that $Q_j+mI, 1 \leq j\leq q $ are invertible for all integers $m\geq 0$. \\

In the present work, we consider an Euler-type integral and present some integral representations of ${}_3F_2(\cdot)$ matrix function using the suitable adjustment of matrix parameters.

\begin{theorem}\label{thm1}
Let $P, Q$ and $R$ be commuting matrices in $\mathbf{C}^{R\times R}$ such that $Q,R$ and $R-Q$ are positive stable. Then, for $|z|<1$, we have the following integral representation:
\begin{align}\label{thm1a}
{}_3F_2\left(P,\frac{Q}{2},\frac{Q+I}{2};\frac{R}{2},\frac{R+I}{2};z\right) = \frac{\Gamma(R)}{\Gamma(Q)\Gamma(R-Q)} \int_0^1 u^{Q-I} (1-u)^{R-Q-I} (1-zu^2)^{-P} du. 
\end{align}
\end{theorem}

\begin{proof}
From the left-hand side of (\ref{thm1a}), we find that
\begin{align}\label{thm1b}
{}_3F_2\left(P,\frac{Q}{2},\frac{Q+I}{2};\frac{R}{2},\frac{R+I}{2};z\right) = \sum_{m=0}^\infty \frac{\left(P\right)_m \left(\frac{Q}{2}\right)_m\left(\frac{Q+I}{2}\right)_m}{\left(\frac{R}{2}\right)_m\left(\frac{R+I}{2}\right)_m} \frac{z^m}{m!}.
\end{align}
Using the relation 
\begin{align*}
\left(P\right)_{2m} = 2^{2m} \left(\frac{P}{2}\right)_m\left(\frac{P+I}{2}\right)_m,
\end{align*}
equation (\ref{thm1b}) becomes
\begin{align*}
{}_3F_2\left(P,\frac{Q}{2},\frac{Q+I}{2};\frac{R}{2},\frac{R+I}{2};z\right) &= \sum_{m=0}^\infty \frac{\left(P\right)_m \left(Q\right)_{2m}}{\left(R\right)_{2m}} \frac{z^m}{m!} \\
&= \frac{\Gamma(R)}{\Gamma(Q)\Gamma(R-Q)}  \int_0^1 u^{Q-I} (1-u)^{R-Q-I} \sum_{m=0}^\infty \frac{\left(P\right)_m (zu^2)^m}{m!} \ du \\
&= \frac{\Gamma(R)}{\Gamma(Q)\Gamma(R-Q)} \int_0^1 u^{Q-I} (1-u)^{R-Q-I} (1-zu^2)^{-P} du. 
\end{align*}
This completes the proof of Theorem \ref{thm1}.
\end{proof}

\begin{corollary}
Under the conditions stated in Theorem \ref{thm1}, the following integral relation holds true:
\begin{align}
{}_3F_2\left(-kI,\frac{Q}{2},\frac{Q+I}{2};\frac{R}{2},\frac{R+I}{2};z\right) = \frac{\Gamma(R)}{\Gamma(Q)\Gamma(R-Q)} \int_0^1 u^{Q-I} (1-u)^{R-Q-I} (1-zu^2)^{kI} du. 
\end{align}
\end{corollary}

\begin{theorem}\label{thm2}
Let $P, Q$ and $R$ be commuting matrices in $\mathbf{C}^{R\times R}$ such that $R,R-P,R-Q$ and $R-Q-P$ are positive stable. Then the following integral representation holds true:
\begin{align}\label{thm2a}
{}_3F_2\left(P,\frac{Q}{2},\frac{Q+I}{2};\frac{R}{2},\frac{R+I}{2};1\right) = \frac{\Gamma(R)\Gamma(R-Q-P)}{\Gamma(R-P)\Gamma(R-Q)} \ {}_2F_1\left(P, Q; R-P;-1\right). 
\end{align}
\end{theorem}

\begin{proof}
Substituting $z=1$ in (\ref{thm1a}), it becomes
\begin{align*}
{}_3F_2\left(P,\frac{Q}{2},\frac{Q+I}{2};\frac{R}{2},\frac{R+I}{2};1\right) &= \frac{\Gamma(R)}{\Gamma(Q)\Gamma(R-Q)} \int_0^1 u^{Q-I} (1-u)^{R-Q-P-I} (1+u)^{-P} du \\
&= \frac{\Gamma(R)}{\Gamma(Q)\Gamma(R-Q)} \sum_{m=0}^\infty {-P \choose m} \int_0^1 u^{Q+(m-1)I} (1-u)^{R-Q-P-I} \ du \\
&= \frac{\Gamma(R)}{\Gamma(Q)\Gamma(R-Q)} \sum_{m=0}^\infty {-P \choose m} \frac{\Gamma(Q+mI)\Gamma(R-Q-P)}{\Gamma(R-P+mI)} \\
&= \frac{\Gamma(R)\Gamma(R-Q-P)}{\Gamma(R-P)\Gamma(R-Q)} \sum_{m=0}^\infty \frac{(P)_m (Q)_m}{(R-P)_m} \frac{(-1)^m}{m!} \\
&= \frac{\Gamma(R)\Gamma(R-Q-P)}{\Gamma(R-P)\Gamma(R-Q)} \ {}_2F_1\left(P, Q; R-P;-1\right),
\end{align*}
which completes the proof of Theorem \ref{thm2}.
\end{proof}

\begin{corollary}\label{cor1}
Under the conditions stated in Theorem \ref{thm2}, the following integral relation holds true:
\begin{align*}
{}_3F_2\left(-nI,\frac{Q}{2},\frac{Q+I}{2};\frac{R}{2},\frac{R+I}{2};1\right) = \frac{\left(R-Q\right)_n}{\left(R\right)_n} \ {}_2F_1\left(-nI, Q; R+nI;-1\right). 
\end{align*}
\end{corollary}

\begin{theorem}\label{thm3}
Let $P, Q$ and $R$ be commuting matrices in $\mathbf{C}^{R\times R}$ such that $Q,R$ and $R-Q$ are positive stable. Then the following integral representation holds true:
\begin{align}\label{thm3a}
{}_3F_2\left(P,\frac{Q}{2},\frac{Q+I}{2};\frac{R}{2},\frac{R+I}{2};\frac{1}{2}\right) = 2^P \sum_{m=0}^\infty {-P \choose m} \frac{(R-Q)_m}{(R)_m} \times {}_2F_1\left(-mI,Q;R+mI;-1\right). 
\end{align}
\end{theorem}
\begin{proof}
Substituting $z=\frac{1}{2}$ in Theorem \ref{thm1}, we obtain
\begin{align*}
&{}_3F_2\left(P,\frac{Q}{2},\frac{Q+I}{2};\frac{R}{2},\frac{R+I}{2};\frac{1}{2}\right) \\ & \quad = \frac{\Gamma(R)}{\Gamma(Q)\Gamma(R-Q)} \int_0^1 u^{Q-I} (1-u)^{R-Q-I} \left(1-\frac{1}{2} u^2\right)^{-P} du\\
& \quad =  \frac{2^P \ \Gamma(R)}{\Gamma(Q)\Gamma(R-Q)} \int_0^1 u^{Q-I} (1-u)^{R-Q-I} \left(2-u^2\right)^{-P} du \\
& \quad =  \frac{2^P \ \Gamma(R)}{\Gamma(Q)\Gamma(R-Q)} \sum_{m=0}^\infty {-P \choose m} \int_0^1 u^{Q-I} (1-u)^{R-Q+(m-1)I} \left(1+u\right)^{m} du \\
& \quad =  \frac{2^P \ \Gamma(R)}{\Gamma(Q)\Gamma(R-Q)} \sum_{m=0}^\infty \sum_{k=0}^m {-P \choose m} {m \choose k} \int_0^1 u^{Q+(k-1)I} (1-u)^{R-Q+(m-1)I} du \\
& \quad = 2^P \sum_{m=0}^\infty \sum_{k=0}^m {-P \choose m} {m \choose k} \frac{(Q)_k(R-Q)_m}{(R)_{k+m}}.
\end{align*}
Using the transformation $(R)_{k+m} = (R)_m (R+m)_k$, we arrive at
\begin{align*}
&{}_3F_2\left(P,\frac{Q}{2},\frac{Q+I}{2};\frac{R}{2},\frac{R+I}{2};\frac{1}{2}\right) \\ & \quad = 2^P \sum_{m=0}^\infty {-P \choose m} \frac{(R-Q)_m}{(R)_m} \sum_{k=0}^m {m \choose k} \frac{(Q)_k}{(R+m)_k}\\
& \quad = 2^P \sum_{m=0}^\infty {-P \choose m} \frac{(R-Q)_m}{(R)_m} \times {}_2F_1\left(-mI,Q;R+mI;-1\right).
\end{align*}
This completes the proof of Theorem \ref{thm3}.
\end{proof}
\begin{corollary}\label{cor2}
Under the conditions stated in Theorem \ref{thm3}, the following integral relation holds true:
\begin{align}
{}_3F_2\left(P,\frac{Q}{2},\frac{Q+I}{2};\frac{R}{2},\frac{R+I}{2};\frac{1}{2}\right) = 2^P \sum_{m=0}^\infty {-P \choose m} \ {}_3F_2\left(-mI,\frac{Q}{2},\frac{Q+I}{2};\frac{R}{2},\frac{R+I}{2};1\right)
\end{align}
\end{corollary}

Now we generalize the Theorem \ref{thm1} in the following form using the suitable adjustment of argument in the ${}_3F_2(\cdot)$ matrix function.
\begin{theorem}\label{thm6}
Let $P, Q$ and $R$ be commuting matrices in $\mathbf{C}^{R\times R}$ such that $Q,R$ and $R-Q$ are positive stable and $w\in\mathbf{R} \backslash \left\lbrace 0,-1\right\rbrace$. Then the following integral representation holds true:
\begin{align}\nonumber
& {}_3F_2\left(P,\frac{Q}{2},\frac{Q+I}{2};\frac{R}{2},\frac{R+I}{2};\frac{1}{w+1}\right) \\ & \qquad\qquad\qquad = \left(\frac{w+1}{w}\right)^P \sum_{m=0}^\infty {-P \choose m} w^{-m} {}_3F_2\left(-mI,\frac{Q}{2},\frac{Q+I}{2};\frac{R}{2},\frac{R+I}{2};1\right). \label{thm6a}
\end{align}
\end{theorem}

\begin{proof}
Substituting $z = \frac{1}{w+1}$ in Theorem \ref{thm1}, we have
\begin{align*}
& {}_3F_2\left(P,\frac{Q}{2},\frac{Q+I}{2};\frac{R}{2},\frac{R+I}{2};\frac{1}{w+1}\right) \\ &\quad = \frac{\Gamma(R)}{\Gamma(Q)\Gamma(R-Q)} \int_0^1 u^{Q-I} (1-u)^{R-Q-I} \left(1-\frac{u^2}{w+1}\right)^{-P} du \\
&\quad = \frac{(w+1)^P \ \Gamma(R)}{\Gamma(Q)\Gamma(R-Q)} \int_0^1 u^{Q-I} (1-u)^{R-Q-I} \left(w+1-u^2\right)^{-P} du \\
&\quad = \left(\frac{w+1}{w}\right)^P \frac{\Gamma(R)}{\Gamma(Q)\Gamma(R-Q)} \int_0^1 u^{Q-I} (1-u)^{R-Q-I} \left(1+\frac{1-u^2}{w}\right)^{-P} du \\
&\quad = \left(\frac{w+1}{w}\right)^P \frac{\Gamma(R)}{\Gamma(Q)\Gamma(R-Q)} \int_0^1 u^{Q-I} (1-u)^{R-Q-I} \sum_{m=0}^\infty {-P \choose m} \left(\frac{1-u^2}{w}\right)^m du \\
&\quad = \left(\frac{w+1}{w}\right)^P \frac{\Gamma(R)}{\Gamma(Q)\Gamma(R-Q)} \sum_{m=0}^\infty {-P \choose m} w^{-m}\int_0^1 u^{Q-I} (1-u)^{R-Q+(m-1)I}  \left(1+u\right)^m du \\
&\quad = \left(\frac{w+1}{w}\right)^P \frac{\Gamma(R)}{\Gamma(Q)\Gamma(R-Q)} \sum_{m=0}^\infty \sum_{k=0}^m {-P \choose m} {m \choose k} w^{-m}\int_0^1 u^{Q+(k-1)I} (1-u)^{R-Q+(m-1)I} du \\
&\quad = \left(\frac{w+1}{w}\right)^P \sum_{m=0}^\infty \sum_{k=0}^m {-P \choose m} {m \choose k} w^{-m} \frac{(Q)_k (R-Q)_m}{(R)_m (R+m)_k}\\
&\quad = \left(\frac{w+1}{w}\right)^P \sum_{m=0}^\infty {-P \choose m} w^{-m} \frac{(R-Q)_m}{(R)_m} \sum_{k=0}^m \frac{(-1)^k (-m)_k}{k!} \frac{(Q)_k}{(R+m)_k}  \\
&\quad = \left(\frac{w+1}{w}\right)^P \sum_{m=0}^\infty {-P \choose m} w^{-m} \frac{(R-Q)_m}{(R)_m} {}_2F_1\left(-mI,Q;R+mI;-1 \right).
\end{align*}
Using Corollary \ref{cor1}, this yields the right hand side of Theorem \ref{thm6}.
\end{proof}

\begin{remark}
It is interesting to observe that for $w=1$, Theorem \ref{thm6} reduces to Theorem \ref{thm3} and Corollary \ref{cor2} and for $w=-2$, it reduces to the following result asserted by Corollary \ref{cor3}.
\end{remark}

\begin{corollary}\label{cor3}
Under the conditions stated in Theorem \ref{thm6}, the following integral relation holds true:
\begin{align*}
& {}_3F_2\left(P,\frac{Q}{2},\frac{Q+I}{2};\frac{R}{2},\frac{R+I}{2};-1\right) = \left(\frac{1}{2}\right)^P \sum_{m=0}^\infty {-P \choose m} (-2)^{-m} {}_3F_2\left(-mI,\frac{Q}{2},\frac{Q+I}{2};\frac{R}{2},\frac{R+I}{2};1\right).
\end{align*}
\end{corollary}

Next we generalize the Theorem \ref{thm1} in the obvious way using the suitable adjustment of matrix parameters by introducing the sequence of $q$ parameters in the ${}_3F_2(\cdot)$ matrix function.

\begin{theorem}\label{thm4}
Let $P, Q$ and $R$ be commuting matrices in $\mathbf{C}^{R\times R}$ such that $Q,R$ and $R-Q$ are positive stable. Then, for $|z|<1$, the following integral representation holds true:
\begin{align}\nonumber
&{}_{q+1}F_q\left(P,\frac{Q}{q},\frac{Q+I}{q},\dots,\frac{Q+(q-1)I}{q};\frac{R}{q},\frac{R+I}{q},\dots,\frac{R+(q-1)I}{q};z\right) \\ & \qquad\qquad\qquad\qquad\quad = \frac{\Gamma(R)}{\Gamma(Q)\Gamma(R-Q)} \int_0^1 u^{Q-I} (1-u)^{R-Q-I} (1-zu^q)^{-P} du. \label{thm4a}
\end{align}
\end{theorem}
\begin{proof}
Using the following relation:
\begin{align*}
\left(P\right)_{mn} = m^{mn} \left(\frac{P}{m}\right)_n \left(\frac{P+I}{m}\right)_n \dots \left(\frac{P+(m-1)I}{m}\right)_n,
\end{align*}
we can easily proceed for the proof similar to Theorem \ref{thm1}. 
\end{proof}
Now for $q=3$ and $z=1$, Theorem \ref{thm4} leads to the following result asserted by Theorem \ref{thm5}.
\begin{theorem}\label{thm5}
Let $P, Q$ and $R$ be commuting matrices in $\mathbf{C}^{R\times R}$ such that $R,R-P,R-Q$ and $R-Q-P$ are positive stable. Then the following integral representation holds true:
\begin{align}\nonumber
& {}_4F_3\left(P,\frac{Q}{3},\frac{Q+I}{3},\frac{Q+2I}{3};\frac{R}{3},\frac{R+I}{3},\frac{R+2I}{3};1\right) \\ & \qquad = \frac{\Gamma(R)\Gamma(R-Q-P)}{\Gamma(R-P)\Gamma(R-Q)} \sum_{m=0}^\infty \frac{(-1)^m (P)_m (Q)_m}{m! (R-P)_m} \ {}_2F_1\left(-mI,Q+mI;R-P+mI;-1\right). \label{thm5a}
\end{align}
\end{theorem}

\begin{proof}
Substituting $q=3$ and $z=1$ in Theorem \ref{thm4}, we have
\begin{align*}
& {}_4F_3\left(P,\frac{Q}{3},\frac{Q+I}{3},\frac{Q+2I}{3};\frac{R}{3},\frac{R+I}{3},\frac{R+2I}{3};1\right) \\ & \quad= \frac{\Gamma(R)}{\Gamma(Q)\Gamma(R-Q)} \int_0^1 u^{Q-I} (1-u)^{R-Q-I} (1-u^3)^{-P}du \\
& \quad= \frac{\Gamma(R)}{\Gamma(Q)\Gamma(R-Q)} \int_0^1 u^{Q-I} (1-u)^{R-Q-P-I} (1+u+u^2)^{-P} du \\ 
& \quad= \frac{\Gamma(R)}{\Gamma(Q)\Gamma(R-Q)} \sum_{m=0}^\infty \frac{(P)_m (-1)^m}{m!} \int_0^1 u^{Q+(m-1)I} (1-u)^{R-Q-P-I} (1+u)^{m} du \\
& \quad= \frac{\Gamma(R)}{\Gamma(Q)\Gamma(R-Q)} \sum_{m=0}^\infty \sum_{k=0}^m \frac{(P)_m (-1)^m}{m!} {m \choose k} \int_0^1 u^{Q+(m+k-1)I} (1-u)^{R-Q-P-I} du \\
& \quad= \frac{\Gamma(R)\Gamma(R-Q-P)}{\Gamma(R-P)\Gamma(R-Q)} \sum_{m=0}^\infty \sum_{k=0}^m \frac{(P)_m (-1)^m}{m!} {m \choose k} \frac{(Q)_m (Q+m)_k}{(R-P)_m(R-P+m)_k} \\
& \quad= \frac{\Gamma(R)\Gamma(R-Q-P)}{\Gamma(R-P)\Gamma(R-Q)} \sum_{m=0}^\infty \frac{(-1)^m (P)_m (Q)_m}{m! (R-P)_m} \ {}_2F_1\left(-mI,Q+mI;R-P+mI;-1\right).
\end{align*}
This completes the proof.
\end{proof}

\begin{theorem}\label{thm7}
Let $P, Q$ and $R$ be commuting matrices in $\mathbf{C}^{R\times R}$ such that $Q,R$ and $R-Q$ are positive stable. Then, for $|z|<1$, the following integral representation holds true:
\begin{align}\nonumber
&{}_{q+1}F_q\left(P,\frac{Q}{q},\frac{Q+I}{q},\dots,\frac{Q+(q-1)I}{q};\frac{R}{q},\frac{R+I}{q},\dots,\frac{R+(q-1)I}{q};z\right) \\ & \qquad = \frac{\Gamma(R)}{\Gamma(Q)\Gamma(R-Q)} \sum_{m=0}^\infty {R-Q-I \choose m} \frac{(-1)^m}{Q+mI} \ {}_2F_1\left(P,\frac{Q+mI}{2};\frac{Q+mI}{2}+I;z\right). \label{thm7a}
\end{align}
\end{theorem}

\begin{proof}
From Theorem \ref{thm4}, we have
\begin{align*}
&{}_{q+1}F_q\left(P,\frac{Q}{q},\frac{Q+I}{q},\dots,\frac{Q+(q-1)I}{q};\frac{R}{q},\frac{R+I}{q},\dots,\frac{R+(q-1)I}{q};z\right) \\ & \quad= \frac{\Gamma(R)}{\Gamma(Q)\Gamma(R-Q)} \int_0^1 u^{Q-I} (1-u)^{R-Q-I} (1-zu^q)^{-P} du \\
& \quad= \frac{\Gamma(R)}{\Gamma(Q)\Gamma(R-Q)} \int_0^1 u^{Q-I} \left[ \sum_{m=0}^\infty {R-Q-I \choose m} (-u)^m \right] (1-zu^q)^{-P} du \\
& \quad= \frac{\Gamma(R)}{\Gamma(Q)\Gamma(R-Q)} \sum_{m=0}^\infty {R-Q-I \choose m} (-1)^m \int_0^1 u^{Q+(m-1)I} (1-zu^q)^{-P} du.
\end{align*}
Substitution $s=u^q$ in the integral on the right hand side of above equation yields
\begin{align*}
&{}_{q+1}F_q\left(P,\frac{Q}{q},\frac{Q+I}{q},\dots,\frac{Q+(q-1)I}{q};\frac{R}{q},\frac{R+I}{q},\dots,\frac{R+(q-1)I}{q};z\right) \\ & \quad= \frac{\Gamma(R)}{\Gamma(Q)\Gamma(R-Q)} \sum_{m=0}^\infty {R-Q-I \choose m} \frac{(-1)^m}{q} \frac{\Gamma\left(\frac{Q+mI}{q}\right)\Gamma(1)}{\Gamma\left(\frac{Q+mI}{q}+I\right)} \ {}_2F_1\left(P,\frac{Q+mI}{2};\frac{Q+mI}{2}+I;z\right).
\end{align*}
Now using the Pochhammer matrix symbols (\ref{PCH2}) yields the right hand side of Theorem \ref{thm7}.
\end{proof}

\section*{Statements \& Declarations}

\subsection*{Funding} Not applicable.

\subsection*{Conflicts of interest/Competing interests} The authors declare that they have no competing interests.

\subsection*{Authors contributions}
Both the authors have equally contributed of reading and writing the manuscript.


\end{document}